\documentclass[12pt,reqno,twoside]{amsart}
\usepackage{amsmath}
\usepackage{amssymb,amsfonts,amsthm,indentfirst}
\usepackage{epsfig}
\usepackage{graphicx}
\usepackage{hyperref}

\newtheorem{theorem}{Theorem}[section]
\newtheorem{proposition}{Proposition}[section]

\newtheorem{lemma}{Lemma}[section]
\newtheorem{corollary}{Corollary}[section]

\newcommand{\be}{\begin{equation}}
\newcommand{\ee}{\end{equation}}
\newcommand{\bea}{\begin{eqnarray}}
\newcommand{\eea}{\end{eqnarray}}
\newcommand{\eeas}{\end{eqnarray*}}
\newcommand{\beas}{\begin{eqnarray*}}

\numberwithin{equation}{section}

\begin{document}

\begin{center}
\textbf{On Quasi-Einstein Sequential Warped Product Manifolds} \bigskip

\textbf{Fatma KARACA and Cihan \"{O}ZG\"{U}R}
\end{center}

\bigskip \textbf{Abstract.} We find the necessary conditions for a
sequential warped product manifold to be a quasi-Einstein manifold. We also
investigate the necessary and sufficient conditions for a sequential
standard static space-time and a sequential generalized Robertson-Walker
space-time to be a manifold of quasi-constant curvature.



\medskip

\section{\textbf{Introduction}\label{sect-introduction}}

A Riemannian manifold $(M,g),$ $n\geq 2,$ is said to be an \textit{Einstein
manifold} \cite{Oneil} if its Ricci tensor $Ric$ is the form of $Ric=\alpha
g,$ where $\alpha $ $\in $ $C^{\infty }\left( M\right) .$ It is well-known
that if $n>2$ then $\alpha $ is a constant. Let $(M,g),$ $n>2,$ be a
Riemannian manifold, then the manifold $(M,g)$ is defined to be a \textit{%
quasi-Einstein manifold} \cite{CM} if the condition
\begin{equation}
Ric(X,Y)=\alpha g(X,Y)+\beta A(X)A(Y)  \label{quasieinstein}
\end{equation}%
is satisfied on $M$, where $\alpha $ and $\beta $ are some differentiable
functions on $M$ with $\beta \neq 0$ and $A$ is a non-zero 1-form such that

\begin{equation}
g(X,U)=A(X),  \label{1form}
\end{equation}%
for every vector field $X$ and $U\in \chi (M)$ being a unit vector field. If
$\beta =0$, then the manifold turns into an Einstein manifold. The notion of
quasi-Einstein manifolds arose during the study of exact solutions of the
Einstein field equations as well as during considerations of quasi-umbilical
hypersurfaces of conformally flat spaces \cite{Gl}.

Let $(M,g),$ $n>2,$ be a Riemannian manifold, then the manifold $(M,g)$ is
defined to be \textit{a quasi-constant curvature} \cite{Chen-Yano} if its
curvature tensor field $R$ satisfies

\begin{equation*}
R(X,Y,Z,W)=a\left[ g(Y,Z)g(X,W)-g(X,Z)g(Y,W)\right]
\end{equation*}%
\begin{equation*}
+b\left[ g(X,W)A(Y)A(Z)-g(X,Z)A(Y)A(W)\right.
\end{equation*}%
\begin{equation*}
\left. +g(Y,Z)A(X)A(W)-g(Y,W)A(X)A(Z)\right]
\end{equation*}%
where $a,b$ are real valued functions on $M$ such that $b\neq 0$ and $A$ is
a 1-form defined by (\ref{1form}). If $b=0,$ then $M$ is a space of constant
curvature. It is easy to see that every Riemannian manifold of
quasi-constant curvature is a quasi-Einstein manifold.

The notion of a warped product was introduced by Bishop and O'Neill in \cite%
{BO} to construct Riemannian manifolds with negative sectional curvature.
Warped products have remarkable applications in differential geometry as
well as in mathematical physics, especially in general relativity.

Warped products have some generalizations like doubly warped products and
multiply warped products. \ A new generalization of warped products is named
as \textit{sequential warped product} \cite{DSU} where the base factor of
the warped product is itself a new warped product manifold. Let $\left(
M_{i},g_{i}\right) $ be $m_{i}$-dimensional Riemannian manifolds,
respectively, where $i\in \left\{ 1,2,3\right\} $ and also $\overline{M}%
=\left( M_{1}\times _{f}M_{2}\right) \times _{h}M_{3}$ be an $%
m=m_{1}+m_{2}+m_{3}$ dimensional Riemannian manifold. Let $%
f:M_{1}\rightarrow \left( 0,\infty \right) $ and $h:M_{1}\times
M_{2}\rightarrow \left( 0,\infty \right) $ be two smooth positive functions
on $M_{1}$ and $M_{1}\times M_{2},$ respectively. The sequential warped
product manifold is the triple product manifold $\overline{M}=\left(
M_{1}\times _{f}M_{2}\right) \times _{h}M_{3}$ equipped with the metric
tensor%
\begin{equation}
\overline{g}=g_{1}\oplus f^{2}g_{2}\oplus h^{2}g_{3},  \label{metric}
\end{equation}%
where $f:M_{1}\rightarrow \left( 0,\infty \right) $ and $h:M_{1}\times
M_{2}\rightarrow \left( 0,\infty \right) $ are warping functions \cite{DSU}.
Taub-Nut and stationary metrics, Schwarzschild and generalized Riemannian
anti de Sitter $\mathbb{T}^{2}$ black hole metrics are non-trivial examples
of sequential warped products \cite{DSU}.

We shall denote $\overline{\nabla },$ $\nabla ^{1}$, $\nabla ^{2}$, $\nabla
^{3}$, $\overline{Ric}$, $Ric^{1},$ $Ric^{2},$ $Ric^{3}$ $,$ $\overline{scal}%
,$ $scal_{1},scal_{2}$ and $scal_{3}$ the Levi-Civita connections, the Ricci
curvatures and the scalar curvatures of the $\overline{M},$ $M_{1},$ $M_{2}$
and $M_{3}$, respectively.

Now, we give the following lemmas:

\begin{lemma}
\label{lemma1}\cite{DSU} Let $\overline{M}=\left( M_{1}\times
_{f}M_{2}\right) \times _{h}M_{3}$ be a sequential warped product with
metric $\overline{g}=g_{1}\oplus f^{2}g_{2}\oplus h^{2}g_{3}$. If $X_{i},$ $%
Y_{i}\in \chi \left( M_{i}\right) $ for $i\in \left\{ 1,2,3\right\} $ then

$i)$ $\overline{\nabla }_{X_{1}}Y_{1}=\nabla _{X_{1}}^{1}Y_{1},$

$ii)$ $\overline{\nabla }_{X_{1}}X_{2}=\overline{\nabla }%
_{X_{2}}X_{1}=X_{1}(\ln f)X_{2},$

$iii)$ $\overline{\nabla }_{X_{2}}Y_{2}=\nabla
_{X_{2}}^{2}Y_{2}-fg_{2}(X_{2},Y_{2})grad^{1}f,$

$iv)$ $\overline{\nabla }_{X_{3}}X_{1}=\overline{\nabla }%
_{X_{1}}X_{3}=X_{1}(\ln h)X_{3},$

$v)$ $\overline{\nabla }_{X_{2}}X_{3}=\overline{\nabla }_{X_{3}}X_{2}=X_{2}(%
\ln h)X_{3},$

$vi)$ $\overline{\nabla }_{X_{3}}Y_{3}=\nabla
_{X_{3}}^{3}Y_{3}-hg_{3}(X_{3},Y_{3})gradh.$
\end{lemma}

\begin{lemma}
\label{lemma2}\cite{DSU} Let $\overline{M}=\left( M_{1}\times
_{f}M_{2}\right) \times _{h}M_{3}$ be a sequential warped product with
metric $\overline{g}=g_{1}\oplus f^{2}g_{2}\oplus h^{2}g_{3}$. If $X_{i},$ $%
Y_{i},$ $Z_{i}\in \chi \left( M_{i}\right) $, then

$i)$ $\overline{R}(X_{1},Y_{1})Z_{1}=R^{1}(X_{1},Y_{1})Z_{1},$

$ii)$ $\overline{R}(X_{2},Y_{2})Z_{2}=R^{2}(X_{2},Y_{2})Z_{2}-\left\Vert
grad^{1}f\right\Vert ^{2}\left[ g_{2}\left( X_{2},Z_{2}\right)
Y_{2}-g_{2}\left( Y_{2},Z_{2}\right) X_{2}\right] ,$

$iii)$ $\overline{R}(X_{1},Y_{2})Z_{1}=-\frac{1}{f}%
H_{1}^{f}(X_{1},Z_{1})Y_{2},$

$iv)$ $\overline{R}(X_{1},Y_{2})Z_{2}=fg_{2}\left( Y_{2},Z_{2}\right) \nabla
_{X_{1}}^{1}grad^{1}f,$

$v)$ $\overline{R}(X_{1},Y_{2})Z_{3}=0,$

$vi)$ $\overline{R}(X_{i},Y_{i})Z_{j}=0$ for $i\neq j,$

$vii)$ $\overline{R}(X_{i},Y_{3})Z_{j}=-\frac{1}{h}H^{h}(X_{i},Z_{j})Y_{3}$
for $i,j=1,2,$

$viii)$ $\overline{R}(X_{i},Y_{3})Z_{3}=hg_{3}\left( Y_{3},Z_{3}\right)
\nabla _{X_{1}}gradh$ for $i=1,2,$

$ix)$ $\overline{R}(X_{3},Y_{3})Z_{3}=R^{3}(X_{3},Y_{3})Z_{3}-\left\Vert
gradh\right\Vert ^{2}\left[ g_{3}\left( X_{3},Z_{3}\right) Y_{3}-g_{3}\left(
Y_{3},Z_{3}\right) X_{3}\right] .$
\end{lemma}

\begin{lemma}
\label{lemma3}\cite{DSU} Let $\overline{M}=\left( M_{1}\times
_{f}M_{2}\right) \times _{h}M_{3}$ be a sequential warped product with
metric $\overline{g}=g_{1}\oplus f^{2}g_{2}\oplus h^{2}g_{3}$. If $X_{i},$ $%
Y_{i}\in \chi \left( M_{i}\right) $, then

$i)$ $\overline{Ric}(X_{1},Y_{1})=$ $Ric^{1}(X_{1},Y_{1})-\frac{m_{2}}{f}%
H_{1}^{f}(X_{1},Y_{1})-\frac{m_{3}}{h}H^{h}(X_{1},Y_{1}),$

$ii)$ $\overline{Ric}(X_{2},Y_{2})=Ric^{2}(X_{2},Y_{2})-\left( f\Delta
^{1}f+(m_{2}-1)\left\Vert grad^{1}f\right\Vert ^{2}\right) g_{2}\left(
X_{2},Y_{2}\right) -\frac{m_{3}}{h}H^{h}(X_{2},Y_{2}),$

$iii)$ $\overline{Ric}(X_{3},Y_{3})=Ric^{3}(X_{3},Y_{3})-\left( h\Delta
h+(m_{3}-1)\left\Vert gradh\right\Vert ^{2}\right) g_{3}\left(
X_{3},Y_{3}\right) ,$

$iv)$ $\overline{Ric}(X_{i},X_{j})=0$ for $i\neq j.$
\end{lemma}

In \cite{CM}, Chaki and Maity introduced the notion of quasi-Einstein
manifolds. In \cite{DG}, De and Ghosh studied some properties of quasi
Einstein manifolds and considered quasi-Einstein hypersurfaces of Euclidean
spaces. In \cite{SO}, Sular and the second author studied quasi-Einstein
warped product manifolds for arbitrary dimension $n\geq 3$. In \cite{Dumitru}%
, Dumitru studied the expressions of the Ricci tensors and scalar curvatures
for the bases and fibres on quasi-Einstein warped product manifolds. For
further developments about quasi-Einstein manifolds; see \cite{ADEHM}, \cite%
{DG2005}, \cite{DSS}, \cite{GDB} and \cite{DPB}.\ In \cite{DSU}, De, Shenawy
and \"{U}nal defined sequential warped product manifolds. In \cite{PP},
Pahan and Pal studied the Einstein sequential warped product space with
negative scalar curvature and gave an example of the Einstein sequential
warped space. In \cite{Sahin}, \c{S}ahin introduced sequential warped
product submanifolds of Kaehler manifolds and obtained some examples.
Motivated by the above studies, in the present paper, we consider
quasi-Einstein sequential warped product manifolds. Firstly, we obtain that
the sequential warped product manifolds are product manifolds if $\alpha <0$%
. Then, we investigate the necessary and sufficient conditions for a
sequential standard static space-time and a sequential generalized
Robertson-Walker space-time to be quasi-constant curvature.

\section{\textbf{The quasi-Einstein sequential warped products \label%
{quasi-Einstein}}}

In this section, we give some results on quasi-Einstein sequential warped
product manifolds. Using Lemma \ref{lemma3}, we obtain the following
proposition:

\begin{proposition}
\label{cor1}Let $\overline{M}=\left( M_{1}\times _{f}M_{2}\right) \times
_{h}M_{3}$ be a quasi-Einstein sequential warped product manifold whose
Ricci tensor is of the form $\overline{Ric}=\alpha \overline{g}+\beta
A\otimes A$. When $U$ is a vector field on $\overline{M},$ then the Ricci
tensors of $M_{1}$, $M_{2}$ and $M_{3}$ satisfy the following equalities:%
\begin{equation*}
Ric^{1}(X_{1},Y_{1})=\alpha g_{1}(X_{1},Y_{1})+\beta
g_{1}(X_{1},U_{1})g_{1}(Y_{1},U_{1})
\end{equation*}%
\begin{equation}
+\frac{m_{2}}{f}H_{1}^{f}(X_{1},Y_{1})+\frac{m_{3}}{h}H^{h}(X_{1},Y_{1}),
\label{i.1}
\end{equation}%
\begin{equation*}
Ric^{2}(X_{2},Y_{2})=\left( \alpha f^{2}+f\Delta ^{1}f+(m_{2}-1)\left\Vert
grad^{1}f\right\Vert ^{2}\right) g_{2}\left( X_{2},Y_{2}\right)
\end{equation*}%
\begin{equation}
+\beta f^{4}g_{2}\left( X_{2},U_{2}\right) g_{2}\left( Y_{2},U_{2}\right) +%
\frac{m_{3}}{h}H^{h}(X_{2},Y_{2})  \label{i.2}
\end{equation}%
and%
\begin{equation*}
Ric^{3}(X_{3},Y_{3})=\left( \alpha h^{2}+h\Delta h+(m_{3}-1)\left\Vert
gradh\right\Vert ^{2}\right) g_{3}\left( X_{3},Y_{3}\right)
\end{equation*}%
\begin{equation}
+\beta h^{4}g_{3}\left( X_{3},U_{3}\right) g_{3}\left( Y_{3},U_{3}\right) .
\label{i.3}
\end{equation}
\end{proposition}

\begin{proof}
Assume that $\overline{M}=\left( M_{1}\times _{f}M_{2}\right) \times
_{h}M_{3}$ is a quasi-Einstein sequential warped product and its Ricci
tensor is $\overline{Ric}=\alpha \overline{g}+\beta A\otimes A$. Decomposing
the vector field $U$ uniquely into its components $U_{1}$, $U_{2}$ and $%
U_{3} $ on $M_{1},$ $M_{2}$ and $M_{3}$, respectively, we can write%
\begin{equation}
U=U_{1}+U_{2}+U_{3}.  \label{a}
\end{equation}%
Using Lemma \ref{lemma3}, the equations (\ref{quasieinstein}), (\ref{metric}%
) and (\ref{a}), we obtain the equations (\ref{i.1}), (\ref{i.2}) and (\ref%
{i.3}). This completes the proof.
\end{proof}

For the scalar curvatures, using Proposition \ref{cor1}, we have the
following corollary:

\begin{corollary}
\label{scalar}Let $\overline{M}=\left( M_{1}\times _{f}M_{2}\right) \times
_{h}M_{3}$ be a quasi-Einstein sequential warped product manifold. When $U$
is a vector field on $\overline{M},$ then the scalar curvature of $M_{1}$, $%
M_{2}$ and $M_{3}$ satisfy the following conditions:%
\begin{equation*}
scal_{1}=\alpha m_{1}+\beta g_{1}(U_{1},U_{1})+\frac{m_{2}}{f}\Delta ^{1}f+%
\frac{m_{3}}{h}\Delta h,
\end{equation*}%
\begin{equation*}
scal_{2}=\left( \alpha f^{2}+f\Delta ^{1}f+(m_{2}-1)\left\Vert
grad^{1}f\right\Vert ^{2}\right) m_{2}+\beta f^{4}g_{2}(U_{2},U_{2})+\frac{%
m_{3}}{h}\Delta h
\end{equation*}%
and%
\begin{equation*}
scal_{3}=\left( \alpha h^{2}+h\Delta h+(m_{3}-1)\left\Vert gradh\right\Vert
^{2}\right) m_{3}+\beta h^{4}g_{3}(U_{3},U_{3}).
\end{equation*}
\end{corollary}

\begin{lemma}
\cite{KK} \label{lemmadiv}Let $f$ be a smooth function on $M_{1}$. Then the
divergence of the Hessian tensor $H^{f}$ satisfies
\begin{equation*}
div\left( H^{f}\right) (X)=Ric\left( gradf,X\right) -\Delta \left( df\right)
\left( X\right) ,
\end{equation*}%
where $\Delta $ denotes the Laplacian on $M_{1}$.
\end{lemma}

Using Proposition \ref{cor1} and Lemma \ref{lemmadiv}, we can state the
following proposition:

\begin{proposition}
\label{cor2}Let $(M_{1},g_{1})$ be a compact Riemannian manifold with $%
m_{1}\geq 2.$ Suppose that $f$ is a nonconstant smooth function on $M_{1}$
satisfying equation $(\ref{i.1})$ for constants $\alpha ,\beta \in
\mathbb{R}
,$ $m_{2}\geq 2$ and $U\in \chi \left( \overline{M}\right) .$ If the
condition
\begin{equation*}
\frac{m_{2}\beta }{f}g_{1}(grad^{1}f,U_{1})g_{1}(X,U_{1})+\frac{m_{2}m_{3}}{%
fh}H^{h}(grad^{1}f,X)+m_{3}div(\frac{H^{h}}{h})(X)
\end{equation*}%
\begin{equation}
=\frac{m_{3}}{2}d(\frac{\Delta h}{h})(X)+\frac{2m_{2}}{f}d\left( \Delta
^{1}f\right) (X)  \label{condition1}
\end{equation}%
holds, then $f$ satisfies
\begin{equation}
\lambda =\alpha f^{2}+f\Delta ^{1}f+(m_{2}-1)\left\Vert grad^{1}f\right\Vert
^{2}  \label{lambda1}
\end{equation}%
for a constant $\lambda \in
\mathbb{R}
.$
\end{proposition}

\begin{proof}
Assume that $U$\ $\in $ $\chi \left( \overline{M}\right) $. By Corollary \ref%
{scalar}, we have%
\begin{equation}
scal_{1}=\alpha m_{1}+\beta g_{1}(U_{1},U_{1})+\frac{m_{2}}{f}\Delta ^{1}f+%
\frac{m_{3}}{h}\Delta h.  \label{a.1}
\end{equation}%
By the use of the second Bianchi identity, we have%
\begin{equation}
d(scal_{1})=2div(Ric^{1}).  \label{a.2}
\end{equation}%
From the equations (\ref{a.1}) and (\ref{a.2}), we get%
\begin{equation}
div(Ric^{1})=\frac{m_{2}}{2f^{2}}\left[ -\left( \Delta ^{1}f\right) \left(
df\right) +fd\left( \Delta ^{1}f\right) \right] +\frac{m_{3}}{2h^{2}}\left[
-\left( \Delta h\right) (dh)+hd\left( \Delta h\right) \right] .  \label{a.3}
\end{equation}%
It is well known from the definition that%
\begin{equation}
div(\frac{1}{f}H_{1}^{f})(X)=-\frac{1}{f^{2}}H_{1}^{f}(grad^{1}f,X)+\frac{1}{%
f}divH_{1}^{f}(X)  \label{a.4}
\end{equation}%
for any vector field $X$ on $M_{1}.$ Since $H_{1}^{f}(grad^{1}f,X)=\frac{1}{2%
}d(\left\Vert grad^{1}f\right\Vert ^{2})(X)$, the equation (\ref{a.4}) turns
into%
\begin{equation}
div(\frac{1}{f}H_{1}^{f})(X)=-\frac{1}{2f^{2}}d(\left\Vert
grad^{1}f\right\Vert ^{2})(X)+\frac{1}{f}divH_{1}^{f}(X).  \label{a.5}
\end{equation}%
Using Lemma \ref{lemmadiv} and equation (\ref{i.1}) in (\ref{a.5}), we find%
\begin{equation*}
div(\frac{1}{f}H_{1}^{f})(X)=\frac{1}{2f^{2}}\left[ \left( m_{2}-1\right)
d(\left\Vert grad^{1}f\right\Vert ^{2})(X)+(2\alpha f)df(X)-(2f)d\left(
\Delta ^{1}f\right) (X)\right]
\end{equation*}%
\begin{equation}
+\frac{\beta }{f}g_{1}(grad^{1}f,U_{1})g_{1}(X,U_{1})+\frac{m_{3}}{fh}%
H^{h}(grad^{1}f,X).  \label{a.6}
\end{equation}%
From equation (\ref{i.1}), we can write%
\begin{equation}
div(Ric^{1})(X)=m_{2}div(\frac{1}{f}H_{1}^{f})(X)+m_{3}div(\frac{1}{h}%
H^{h})(X).  \label{a.7}
\end{equation}%
Using the equation (\ref{a.6}) into (\ref{a.7}), we obtain%
\begin{equation*}
div(Ric^{1})(X)=\frac{m_{2}}{2f^{2}}\left[ \left( m_{2}-1\right)
d(\left\Vert grad^{1}f\right\Vert ^{2})(X)+(2\alpha f)df(X)-(2f)d\left(
\Delta ^{1}f\right) (X)\right]
\end{equation*}%
\begin{equation}
+m_{2}\frac{\beta }{f}g_{1}(grad^{1}f,U_{1})g_{1}(X,U_{1})+\frac{m_{2}m_{3}}{%
fh}H^{h}(grad^{1}f,X)+m_{3}div(\frac{1}{h}H^{h})(X).  \label{a.8}
\end{equation}%
From the equations (\ref{a.3}) and (\ref{a.8}), we get%
\begin{equation*}
\frac{m_{2}}{2f^{2}}d\left[ \left( m_{2}-1\right) (\left\Vert
grad^{1}f\right\Vert ^{2})+(\alpha f^{2})+f\Delta ^{1}f\right] (X)
\end{equation*}%
\begin{equation*}
+m_{2}\frac{\beta }{f}g_{1}(grad^{1}f,U_{1})g_{1}(X,U_{1})-\frac{2m_{2}}{f}%
d(\Delta ^{1}f)(X)+\frac{m_{2}m_{3}}{fh}H^{h}(grad^{1}f,X)
\end{equation*}%
\begin{equation}
+m_{3}div(\frac{1}{h}H^{h})(X)-\frac{m_{3}}{2}d\left( \frac{\Delta h}{h}%
\right) (X)=0.  \label{a.9}
\end{equation}%
Using the condition (\ref{condition1}), we have%
\begin{equation*}
d\left[ \left( m_{2}-1\right) (\left\Vert grad^{1}f\right\Vert ^{2})+\alpha
f^{2}+f\Delta ^{1}f\right] (X)=0.
\end{equation*}%
Therefore, we obtain the equation (\ref{lambda1}). This completes the proof.
\end{proof}

From \cite{PP}, we have the following lemma:

\begin{lemma}
\label{lemmadiv1}Let $h$ be a smooth function on $M_{1}\times M_{2}$. Then
the divergence of the Hessian tensor $H^{h}$ satisfies
\begin{equation*}
div\left( H^{h}\right) (X)=Ric\left( gradh,X\right) -\Delta \left( dh\right)
\left( X\right) ,
\end{equation*}%
where $\Delta $ denotes the Laplacian on $M_{1}\times M_{2}$.
\end{lemma}

Using the same method in the proof of Proposition \ref{cor2}, we can state
the following proposition:

\begin{proposition}
\label{pro1}Let $(M_{1},g_{1})$ and $(M_{2},g_{2})$ be two compact
Riemannian manifold with $m_{1}\geq 2$ and $m_{2}\geq 2.$ Suppose that $h$
is a nonconstant smooth function on $M_{1}\times M_{2}$ satisfying%
\begin{equation}
Ric^{2}=\lambda g_{2}(X_{2},Y_{2})+\beta f^{4}g_{2}\left( X_{2},U_{2}\right)
g_{2}\left( Y_{2},U_{2}\right) +\frac{m_{3}}{h}H^{h}(X_{2},Y_{2})
\label{b.1}
\end{equation}%
for a constant $\lambda \in
\mathbb{R}
$ and $U\in \chi \left( \overline{M}\right) .$ If the condition
\begin{equation*}
\frac{m_{3}}{h}(\lambda -\alpha )dh+\beta div(f^{4})g_{2}\left(
X_{2},U_{2}\right) g_{2}\left( Y_{2},U_{2}\right) +\frac{m_{3}\beta }{h}%
f^{4}g_{2}\left( gradh,U_{2}\right) g_{2}\left( X,U_{2}\right)
\end{equation*}%
\begin{equation}
=\frac{2m_{3}}{h}d(\Delta h)+2\beta f^{3}dfg_{2}\left( U_{2},U_{2}\right)
\label{condition2}
\end{equation}%
holds, then $h$ satisfies
\begin{equation}
\nu =\alpha h^{2}+h\Delta h+(m_{3}-1)\left\Vert gradh\right\Vert ^{2}
\label{nu1}
\end{equation}%
for a constant $\nu \in
\mathbb{R}
.$ Hence for compact quasi-Einstein spaces $\left( M_{3},g_{3}\right) $ of
dimension $m_{3}\geq 2$ with $Ric^{3}=\nu g_{3}+\beta A\otimes A$, we can
make a quasi-Einstein sequential warped product space $\overline{M}=\left(
M_{1}\times _{f}M_{2}\right) \times _{h}M_{3}$ with $\overline{Ric}=\alpha
\overline{g}+\beta A\otimes A$.
\end{proposition}

\begin{proof}
Assume that $U$ $\in $ $\chi \left( \overline{M}\right) .$ By taking the
trace of both sides of the equation (\ref{b.1}), we get%
\begin{equation}
scal_{2}=\lambda m_{2}+\beta f^{4}g_{2}(U_{2},U_{2})+\frac{m_{3}}{h}\Delta h.
\label{b.2}
\end{equation}%
Using the same method with Proposition \ref{cor2}, with the help of Lemma %
\ref{lemmadiv1}, we have the equation (\ref{nu1}) for a constant $\nu \in
\mathbb{R}
$ if the condition (\ref{condition2}) holds. Thus the first part of the
proposition is proved. For a compact quasi-Einstein manifold $\left(
M_{3},g_{3}\right) $ of dimension $m_{3}\geq 2$ with $Ric^{3}=\nu
g_{3}+\beta A\otimes A,$ we can generate a quasi-Einstein sequential warped
product $\overline{M}=\left( M_{1}\times _{f}M_{2}\right) \times _{h}M_{3}$
with $\overline{Ric}=\alpha \overline{g}+\beta A\otimes A$ by using of
Proposition \ref{cor1} for $U\in \chi \left( \overline{M}\right) $. This
completes the proof.
\end{proof}

Now, we have the following theorem:

\begin{theorem}
Let $\overline{M}=\left( M_{1}\times _{f}M_{2}\right) \times _{h}M_{3}$ be a
compact quasi-Einstein sequential warped product space whose Ricci tensor is
of the form $\overline{Ric}=\alpha \overline{g}+\beta A\otimes A$ with $U\in
\chi \left( \overline{M}\right) $. If $\alpha <0$, then the sequential
warped product becomes a product manifold.
\end{theorem}

\begin{proof}
The equation (\ref{lambda1}) gives us%
\begin{equation}
\lambda =\alpha f^{2}+div(f\Delta ^{1}f)+(m_{2}-2)\left\Vert
grad^{1}f\right\Vert ^{2}.  \label{1}
\end{equation}%
By integrating over $M_{1},$ we have%
\begin{equation}
\lambda =\frac{\alpha }{\vartheta \left( M_{1}\right) }\int%
\limits_{M_{1}}f^{2}+\frac{(m_{2}-2)}{\vartheta \left( M_{1}\right) }%
\int\limits_{M_{1}}\left\Vert grad^{1}f\right\Vert ^{2},  \label{2}
\end{equation}%
where $\vartheta \left( M_{1}\right) $ denotes the volume of $M_{1}.$

\textit{Case }$I$: Assume that $m_{2}\geq 3$. Let $p$ be a maximum point of $%
f$ on $M_{1}.$ Hence, we have $f(p)>0,$ $grad^{1}f(p)=0$ and $\Delta
^{1}f(p)\geq 0.$ Using the equations (\ref{lambda1}) and (\ref{2}), we can
write%
\begin{equation*}
0\leq f\left( p\right) \Delta ^{1}f\left( p\right)
\end{equation*}%
\begin{equation*}
=\alpha f(p)^{2}-\lambda
\end{equation*}%
\begin{equation*}
=\frac{(2-m_{2})}{\vartheta \left( M_{1}\right) }\int\limits_{M_{1}}\left%
\Vert grad^{1}f\right\Vert ^{2}+\frac{\alpha }{\vartheta \left( M_{1}\right)
}\int\limits_{M_{1}}\left( f(p)^{2}-f^{2}\right) \leq 0.
\end{equation*}%
If $\alpha <0,$ then $f$ is a constant.

\textit{Case }$II$: Assume that $m_{2}=2$. Let $q$ be a minimum point of $f$
on $M_{1}.$ Hence, we have $f(q)>0,$ $grad^{1}f(q)=0$ and $\Delta
^{1}f(q)\leq 0.$ Thus, we can write%
\begin{equation*}
0\geq f\left( q\right) \Delta ^{1}f\left( q\right)
\end{equation*}%
\begin{equation*}
=\alpha f(q)^{2}-\lambda
\end{equation*}%
\begin{equation*}
=\frac{(2-m_{2})}{\vartheta \left( M_{1}\right) }\int\limits_{M_{1}}\left%
\Vert grad^{1}f\right\Vert ^{2}+\frac{\alpha }{\vartheta \left( M_{1}\right)
}\int\limits_{M_{1}}\left( f(q)^{2}-f^{2}\right) \geq 0.
\end{equation*}%
If $\alpha <0,$ then $f$ is a constant.

Similarly, the equation (\ref{nu1}) gives us%
\begin{equation}
\nu =\alpha h^{2}+div(h\Delta h)+(m_{3}-2)\left\Vert gradh\right\Vert ^{2}.
\label{3}
\end{equation}%
By integrating over $M_{1}\times M_{2},$ we have%
\begin{equation}
\nu =\frac{\alpha }{\vartheta \left( M_{1}\times M_{2}\right) }%
\int\limits_{M_{1}\times M_{2}}h^{2}+\frac{(m_{3}-2)}{\vartheta \left(
M_{1}\times M_{2}\right) }\int\limits_{M_{1}\times M_{2}}\left\Vert
gradh\right\Vert ^{2},  \label{4}
\end{equation}%
where $\vartheta \left( M_{1}\times M_{2}\right) $ denotes the volume of $%
M_{1}\times M_{2}.$

\textit{Case} $I$: Assume that $m_{3}\geq 3$. Let $\left( p_{1},p_{2}\right)
$ be a maximum point of $h$ on $M_{1}\times M_{2}.$ Hence, we have $h\left(
p_{1},p_{2}\right) >0,$ $gradh\left( p_{1},p_{2}\right) =0$ and $\Delta
h\left( p_{1},p_{2}\right) \geq 0.$ Using the equations (\ref{nu1}) and (\ref%
{4}), we get%
\begin{equation*}
0\leq h\left( p_{1},p_{2}\right) \Delta h\left( p_{1},p_{2}\right)
\end{equation*}%
\begin{equation*}
=\alpha h\left( p_{1},p_{2}\right) ^{2}-\nu
\end{equation*}%
\begin{equation*}
=\frac{(2-m_{3})}{\vartheta \left( M_{1}\times M_{2}\right) }%
\int\limits_{M_{1}\times M_{2}}\left\Vert gradh\right\Vert ^{2}+\frac{\alpha
}{\vartheta \left( M_{1}\times M_{2}\right) }\int\limits_{M_{1}\times
M_{2}}\left( h\left( p_{1},p_{2}\right) ^{2}-h^{2}\right) \leq 0.
\end{equation*}%
If $\alpha <0,$ then $h$ is a constant.

\textit{Case} $II$: Assume that $m_{3}=2$. Let $\left( q_{1},q_{2}\right) $
be a minimum point of $h$ on $M_{1}\times M_{2}.$ Hence, we have $h\left(
q_{1},q_{2}\right) >0,$ $gradh\left( q_{1},q_{2}\right) =0$ and $\Delta
h\left( q_{1},q_{2}\right) \leq 0.$ Thus, we get%
\begin{equation*}
0\geq h\left( q_{1},q_{2}\right) \Delta h\left( q_{1},q_{2}\right)
\end{equation*}%
\begin{equation*}
=\alpha h\left( q_{1},q_{2}\right) ^{2}-\nu
\end{equation*}%
\begin{equation*}
=\frac{(2-m_{3})}{\vartheta \left( M_{1}\times M_{2}\right) }%
\int\limits_{M_{1}\times M_{2}}\left\Vert gradh\right\Vert ^{2}+\frac{\alpha
}{\vartheta \left( M_{1}\times M_{2}\right) }\int\limits_{M_{1}\times
M_{2}}\left( h\left( q_{1},q_{2}\right) ^{2}-h^{2}\right) \geq 0.
\end{equation*}%
If $\alpha <0,$ then $h$ is a constant. This completes the proof.
\end{proof}

Using Corollary \ref{scalar}, we have the following theorem:

\begin{theorem}
\label{Theo1}Let $\overline{M}=\left( M_{1}\times _{f}M_{2}\right) \times
_{h}M_{3}$ be a quasi-Einstein sequential warped product space, where $M_{1}$
and $M_{2}$ are two compact spaces and $M_{3}$ is a compact quasi-Einstein
space. Then the following conditions hold for $U\in \chi \left( \overline{M}%
\right) $:

$i)$ If $scal_{3}\leq 0$, $\alpha ,\beta >0$ and $\Delta h\geq 0$, then $h$
is a constant.

$ii)$ If $m_{2}=1$ and $(\lambda >\alpha f^{2}$or $\lambda <\alpha f^{2})$,
then $f$ is a constant. Thus, $h$ is a constant when $\alpha ,\beta >0$ and $%
\Delta h>0$. Hence, it is a Riemannian product.

$iii)$ If $\left\Vert grad^{1}f\right\Vert \geq \sqrt{\frac{\lambda }{m_{2}-1%
}},$ $\left\Vert gradh\right\Vert \geq \sqrt{\frac{\nu }{m_{3}-1}}$ and $%
\alpha >0,$ then $f$ and $h$ are constants. Hence, it is a Riemannian
product.
\end{theorem}

\begin{proof}
Assume that $U$ $\in $ $\chi \left( \overline{M}\right) $. From Corollary %
\ref{scalar}, we have
\begin{equation}
scal_{1}=\alpha m_{1}+\beta g_{1}(U_{1},U_{1})+\frac{m_{2}}{f}\Delta ^{1}f+%
\frac{m_{3}}{h}\Delta h,  \label{c.1}
\end{equation}%
\begin{equation}
scal_{2}=\lambda m_{2}+\beta f^{4}g_{2}(U_{2},U_{2})+\frac{m_{3}}{h}\Delta h
\label{c.2}
\end{equation}%
and%
\begin{equation}
scal_{3}=\nu m_{3}+\beta h^{4}g_{3}(U_{3},U_{3}).  \label{c.3}
\end{equation}%
$i)$ Using the equation (\ref{c.3}), if $scal_{3}\leq 0$ and $\beta >0$,
then $\nu \leq 0$. From the Proposition \ref{pro1}, we can write%
\begin{equation*}
\alpha h^{2}+h\Delta h=\nu -(m_{3}-1)\left\Vert gradh\right\Vert ^{2}\leq 0.
\end{equation*}%
If $\alpha >0$ and $\Delta h\geq 0$ then, we have%
\begin{equation*}
0\leq h\Delta h\leq -\alpha h^{2}\leq 0.
\end{equation*}%
Therefore, we obtain that $h$ is a constant.

$ii)$ From the Proposition \ref{cor2} and $m_{2}=1$, we have
\begin{equation*}
\lambda -\alpha f^{2}=f\Delta ^{1}f.
\end{equation*}%
Using the above equation, if $\lambda >\alpha f^{2}$or $\lambda <\alpha
f^{2},$ then we find that $f$ is a constant. So we can write%
\begin{equation}
\lambda =\alpha f^{2}.  \label{c.4}
\end{equation}%
By the use of the equations (\ref{c.2}), (\ref{c.4}) and $m_{2}=1,$ we get%
\begin{equation}
\alpha f^{2}+\beta f^{4}g_{2}(U_{2},U_{2})+\frac{m_{3}}{h}\Delta h=0.
\label{c.5}
\end{equation}%
From equation (\ref{c.5}), $\alpha ,\beta >0$ and $\Delta h>0$, we obtain
that $h$ is a constant.

$iii)$ From the Proposition \ref{cor2} and \ref{pro1}, we have%
\begin{equation}
\alpha f^{2}+f\Delta ^{1}f=\lambda -(m_{2}-1)\left\Vert grad^{1}f\right\Vert
^{2}  \label{c.6}
\end{equation}%
and
\begin{equation}
\alpha h^{2}+h\Delta h=\nu -(m_{3}-1)\left\Vert gradh\right\Vert ^{2}.
\label{c.7}
\end{equation}%
Using the conditions $\left\Vert grad^{1}f\right\Vert \geq \sqrt{\frac{%
\lambda }{m_{2}-1}}$, $\left\Vert gradh\right\Vert \geq \sqrt{\frac{\nu }{%
m_{3}-1}}$ and $\alpha >0$ for equations (\ref{c.6}) and (\ref{c.7})
respectively, we get%
\begin{equation*}
0\leq \alpha f^{2}\leq -f\Delta ^{1}f\leq 0
\end{equation*}%
and%
\begin{equation*}
0\leq \alpha h^{2}\leq -h\Delta h\leq 0.
\end{equation*}%
Therefore, we deduce that $f$ and $h$ are constants. This completes the
proof.
\end{proof}

\section{\textbf{Applications}}

The warped products have attracted great attention in differential geometry
and physics because exact solutions to Einstein's equation are obtained with
the help of warped product space-time models \cite{BEE}, \cite{BEP}.
Generalized Robertson-Walker space-times and standard static space-times are
two well known solutions to Einstein's field equations. The standard static
space-times can be considered as a generalization of the Einstein static
universe \cite{DU}. In addition, the Robertson-Walker models in general
relativity describe a simply connected homogeneous isotropic expanding or
contracting universe. Then, generalized Robertson--Walker space-times extend
the Robertson--Walker space-times \cite{DS}. Sequential generalized
Robertson-Walker and sequential standard static space-times are introduced
in \cite{DSU}.

Let $\left( M_{i},g_{i}\right) $ be $m_{i}$-dimensional Riemannian
manifolds, where $i\in \left\{ 1,2,3\right\} .$ Let $f:M_{1}\rightarrow
\left( 0,\infty \right) $ and $h:M_{1}\times M_{2}\rightarrow \left(
0,\infty \right) $ be two smooth positive functions. Then $(m_{1}+m_{2}+1)$%
-dimensional product manifold $\overline{M}=\left( M_{1}\times
_{f}M_{2}\right) \times _{h}I$ equipped with the metric tensor%
\begin{equation}
\overline{g}=\left( g_{1}\oplus f^{2}g_{2}\right) \oplus h^{2}(-dt^{2})
\label{ssst}
\end{equation}%
is a sequential standard static space-time, where $I$ is an open, connected
subinterval of $%
\mathbb{R}
$ and $dt^{2}$ is the Euclidean metric tensor on $I$ \cite{DSU}. In
addition, $(m_{2}+m_{3}+1)$-dimensional product manifold $\overline{M}%
=\left( I\times _{f}M_{2}\right) \times _{h}M_{3}$ equipped with the metric
tensor%
\begin{equation}
\overline{g}=\left( -dt^{2}\oplus f^{2}g_{2}\right) \oplus h^{2}g_{3}
\label{GRW}
\end{equation}%
is a sequential generalized Robertson-Walker space-time, where $I$ is an
open, connected subinterval of $%
\mathbb{R}
$ and $dt^{2}$ is the Euclidean metric tensor on $I$ \cite{DSU}.

\begin{theorem}
Let $\overline{M}=\left( M_{1}\times _{f}M_{2}\right) \times _{h}I$ be a
sequential standard static space-time with metric $\overline{g}=\left(
g_{1}\oplus f^{2}g_{2}\right) \oplus h^{2}(-dt^{2}).$ Then $\left( \overline{%
M},\overline{g}\right) $ is a manifold of quasi-constant curvature if and
only if the following conditions are satisfied

$i)$ $\alpha =\beta h^{2}-\frac{\Delta h}{h},$

$ii)$ $\left( M_{1},g_{1}\right) $ is a quasi-Einstein manifold when the
Hessian tensors $H_{1}^{f}$ and $H^{h}$ satisfy the following forms with $%
U=U_{1}+\partial _{t}$%
\begin{equation*}
H_{1}^{f}(X_{1},Y_{1})=afg_{1}(X_{1},Y_{1})+bfg_{1}(X_{1},U_{1})g_{1}(Y_{1},U_{1})
\end{equation*}%
and%
\begin{equation*}
H^{h}(X_{1},Y_{1})=(-a+h^{2})hg_{1}(X_{1},Y_{1})-bhg_{1}(X_{1},U_{1})g_{1}(Y_{1},U_{1}),
\end{equation*}%
where $a,b\in C^{\infty }\left( \overline{M},%
\mathbb{R}
\right) $ and $b\neq 0.$

$iii)$ $\left( M_{2},g_{2}\right) $ is an Einstein manifold when the Hessian
tensor $H^{h}$ satisfies the following form with $U=U_{1}+\partial _{t}$%
\begin{equation*}
H^{h}(X_{2},Y_{2})=(-a+h^{2})f^{2}hg_{2}(X_{2},Y_{2})
\end{equation*}%
where $a,b\in C^{\infty }\left( \overline{M},%
\mathbb{R}
\right) $ and $b\neq 0.$
\end{theorem}

\begin{proof}
From \cite{DSU}, we have%
\begin{equation}
\overline{Ric}(X_{1},Y_{1})=Ric^{1}(X_{1},Y_{1})-\frac{m_{2}}{f}%
H_{1}^{f}(X_{1},Y_{1})-\frac{1}{h}H^{h}(X_{1},Y_{1}),  \label{d1}
\end{equation}%
\begin{equation*}
\overline{Ric}(X_{2},Y_{2})=Ric^{2}(X_{2},Y_{2})
\end{equation*}%
\begin{equation}
-\left( f\Delta ^{1}f+\left( m_{2}-1\right) \left\Vert grad^{1}f\right\Vert
^{2}\right) g_{2}(X_{2},Y_{2})-\frac{1}{h}H^{h}(X_{2},Y_{2}),  \label{d2}
\end{equation}%
\begin{equation}
\overline{Ric}(\partial _{t},\partial _{t})=h\Delta h.  \label{d3}
\end{equation}%
Let $\left( \overline{M},\overline{g}\right) $ be a manifold of
quasi-constant curvature. Thus, $\left( \overline{M},\overline{g}\right) $
is a quasi-Einstein manifold. Using the equations (\ref{quasieinstein}) and (%
\ref{1form}), we have%
\begin{equation}
\overline{Ric}(\partial _{t},\partial _{t})=-\alpha h^{2}+\beta h^{4}.
\label{d4}
\end{equation}%
From the equations (\ref{d3}) and (\ref{d4}), we find%
\begin{equation*}
\alpha =\beta h^{2}-\frac{\Delta h}{h}.
\end{equation*}%
Using Theorem 4.8 in \cite{DSU}, we obtain
\begin{equation}
H^{h}(X_{1},Y_{1})=(-a+h^{2})hg_{1}(X_{1},Y_{1})-bhg_{1}(X_{1},U_{1})g_{1}(Y_{1},U_{1})
\label{d5}
\end{equation}%
and%
\begin{equation}
H_{1}^{f}(X_{1},Y_{1})=afg_{1}(X_{1},Y_{1})+bfg_{1}(X_{1},U_{1})g_{1}(Y_{1},U_{1})
\label{d6}
\end{equation}%
with $U=U_{1}+\partial _{t}$. By the use of the equations (\ref%
{quasieinstein}), (\ref{d5}) and (\ref{d6}) into (\ref{d1}), we obtain%
\begin{equation*}
Ric^{1}(X_{1},Y_{1})=(\alpha +\left( m_{2}-1\right)
a+h^{2})g_{1}(X_{1},Y_{1})+\left( \beta +\left( m_{2}-1\right) b\right)
g_{1}(X_{1},U_{1})g_{1}(Y_{1},U_{1}).
\end{equation*}%
Hence, $\left( M_{1},g_{1}\right) $ is a quasi-Einstein manifold.

Using Theorem 4.8 in \cite{DSU}, we calculate%
\begin{equation}
H^{h}(X_{2},Y_{2})=(-a+h^{2})f^{2}hg_{2}(X_{2},Y_{2}),  \label{d7}
\end{equation}%
with $U=U_{1}+\partial _{t}$. Using the equations (\ref{quasieinstein}) and (%
\ref{d7}) into (\ref{d2}), we find%
\begin{equation*}
Ric^{2}(X_{2},Y_{2})=\left( \alpha f^{2}+f\Delta ^{1}f+\left( m_{2}-1\right)
\left\Vert grad^{1}f\right\Vert ^{2}-af^{2}+f^{2}h^{2}\right)
g_{2}(X_{2},Y_{2}).
\end{equation*}%
Thus, we obtain that $\left( M_{2},g_{2}\right) $ is an Einstein manifold.
The converse is trivial. This completes the proof.
\end{proof}

\begin{theorem}
Let $\overline{M}=\left( I\times _{f}M_{2}\right) \times _{h}M_{3}$ be a
sequential generalized Robertson-Walker space-time with metric $\overline{g}%
=\left( -dt^{2}\oplus f^{2}g_{2}\right) \oplus h^{2}g_{3}.$ Then $\left(
\overline{M},\overline{g}\right) $ is a manifold of quasi-constant curvature
if and only if the following conditions are satisfied

$i)$ $\beta -\alpha =\frac{m_{2}}{f}f^{\prime \prime }-\frac{m_{3}}{h}\frac{%
\partial ^{2}h}{\partial t^{2}},$

$ii)$ $\left( M_{2},g_{2}\right) $ is a quasi-Einstein manifold when the
Hessian tensor $H^{h}$ satisfies the following form with $U=\partial
_{t}+U_{2},$%
\begin{equation*}
H^{h}(X_{2},Y_{2})=ahf^{2}g_{2}(X_{2},Y_{2})+bhf^{4}g_{2}(X_{2},U_{2})g_{2}(Y_{2},U_{2}),
\end{equation*}%
where $a,b\in C^{\infty }\left( \overline{M},%
\mathbb{R}
\right) $ and $b\neq 0.$

$iii)$ $\left( M_{3},g_{3}\right) $ is an Einstein manifold with $U=\partial
_{t}+U_{2}$.
\end{theorem}

\begin{proof}
From \cite{DSU}, we have%
\begin{equation}
\overline{Ric}(\partial _{t},\partial _{t})=\frac{m_{2}}{f}f^{\prime \prime
}+\frac{m_{3}}{h}\frac{\partial ^{2}h}{\partial t^{2}},  \label{e1}
\end{equation}%
\begin{equation*}
\overline{Ric}(X_{2},Y_{2})=Ric^{2}(X_{2},Y_{2})-\left( -ff^{\prime \prime
}-\left( m_{2}-1\right) \left( f^{\prime }\right) ^{2}\right)
g_{2}(X_{2},Y_{2})
\end{equation*}%
\begin{equation}
-\frac{m_{3}}{h}H^{h}(X_{2},Y_{2}),  \label{e2}
\end{equation}%
\begin{equation}
\overline{Ric}(X_{3},Y_{3})=Ric^{3}(X_{3},Y_{3})-\left( h\Delta
h+(m_{3}-1)\left\Vert gradh\right\Vert ^{2}\right) g_{3}(X_{3},Y_{3}).
\label{e3}
\end{equation}%
Let $\left( \overline{M},\overline{g}\right) $ be a manifold of
quasi-constant curvature. Thus, $\left( \overline{M},\overline{g}\right) $
is a quasi-Einstein manifold. Using the equations (\ref{quasieinstein}) and (%
\ref{1form}), we get%
\begin{equation}
\overline{Ric}(\partial _{t},\partial _{t})=-\alpha +\beta .  \label{e4}
\end{equation}%
Using the equations (\ref{e1}) and (\ref{e4}), we obtain%
\begin{equation*}
\beta -\alpha =\frac{m_{2}}{f}f^{\prime \prime }+\frac{m_{3}}{h}\frac{%
\partial ^{2}h}{\partial t^{2}}.
\end{equation*}%
Using Proposition 4.2 in \cite{DSU}, we obtain
\begin{equation}
H^{h}(X_{2},Y_{2})=ahf^{2}g_{2}(X_{2},Y_{2})+bhf^{4}g_{2}(X_{2},U_{2})g_{2}(Y_{2},U_{2})
\label{e5}
\end{equation}%
with $U=\partial _{t}+U_{2}$. In view of the equations (\ref{quasieinstein})
and (\ref{e5}) into (\ref{e2}), it follows that%
\begin{equation*}
Ric^{2}(X_{2},Y_{2})=\left( \alpha f^{2}-ff^{\prime \prime }-\left(
m_{2}-1\right) \left( f^{\prime }\right) ^{2}+m_{3}af^{2}\right)
g_{2}(X_{2},Y_{2})
\end{equation*}%
\begin{equation*}
+\left( \beta +m_{3}b\right) f^{4}g_{2}(X_{2},U_{2})g_{2}(Y_{2},U_{2}).
\end{equation*}%
Hence, $\left( M_{2},g_{2}\right) $ is a quasi-Einstein manifold. Using the
equations (\ref{quasieinstein}) into (\ref{e3}), we find%
\begin{equation*}
Ric^{3}(X_{3},Y_{3})=\left( \alpha h^{2}+h\Delta h+(m_{3}-1)\left\Vert
gradh\right\Vert ^{2}\right) g_{3}(X_{3},Y_{3}).
\end{equation*}%
Thus, $\left( M_{3},g_{3}\right) $ is an Einstein manifold with $U=\partial
_{t}+U_{2}$. The converse is trivial. This completes the proof.
\end{proof}

\bigskip

Fatma KARACA

Beykent University,

Department of Mathematics,

34550, \.{I}stanbul, Turkey.

E-mail: fatmagurlerr@gmail.com\newline
\newline

Cihan \"{O}ZG\"{U}R

Bal\i kesir University,

Department of Mathematics,

10145, \c{C}a\u{g}\i \c{s}, Bal\i kesir, Turkey.

E-mail: cozgur@balikesir.edu.tr


\begin{thebibliography}{99}

\bibitem{ADEHM} Arslan, K., Deszcz, R., Ezenta\c{s}, R., Hotlos, M.,
Murathan, C., \textit{On generalized Robertson--Walker spacetimes satisfying
some curvature condition}, Turkish J. Math., 38(2), 353-373,
2014.

\bibitem{BEE} Beem, J. K., Ehrlich, P., Easley, K., \textit{Global
Lorentzian Geometry}, 2nd edn, Dekker, New York, 1996.

\bibitem{BEP} Beem, J. K., Ehrlich, P., Powell, T. G., \textit{Warped
product manifolds in relativity Selected Studies}: Physics-Astrophysics,
Mathematics, History of Science (Amsterdam: North Holland), 41-66, 1982.

\bibitem{BO} Bishop, R., O'Neill, B., \textit{Manifolds of negative curvature%
}, Trans. Am. Math. Soc., 145, 1-49, 1969.

\bibitem{CM} Chaki, M. C., Maity, R. K., \textit{On Quasi-Einstein Manifolds},
Publ. Math. Debrecen, 57, 297-306, 2000.

\bibitem{Chen-Yano} Chen, B. Y., Yano, K., \textit{Hypersurfaces of a
conformally flat space}, Tensor (N.S.), 26, 318-322, 1972.

\bibitem{DG} De, U. C., Ghosh, G. C., \textit{On quasi Einstein manifolds},
Period. Math. Hungar., 48, 223-231, 2004.

\bibitem{DG2005} De, U. C., Ghosh, G. C., \textit{Some global properties of
generalized quasi-Einstein manifolds}, Ganita, 56, 65-70, 2005.

\bibitem{DS} De, U. C., Shenawy, S., \textit{Generalized quasi-Einstein GRW
space-times}, International Journal of Geometric Methods in Modern Physics,
16(08), 1950124, 2019.

\bibitem{DSS} De, U. C., Sengupta, J., Saha, D., \textit{Conformally at
quasi-Einstein spaces}, Kyungpook Math. J., 46, 417-423, 2006.

\bibitem{DSU} De, U. C., Shenawy, S., \"{U}nal, B., \textit{Sequential warped
products: curvature and conformal vector fields}, Filomat, 33(13),
4071-4083, 2019.

\bibitem{DU} Dobarro, F., \"{U}nal, B., \textit{Special standard static
space--times}, Nonlinear Analysis: Theory, Methods and Applications, 59(5),
759-770, 2004.

\bibitem{Dumitru} Dumitru, D., \textit{On quasi-Einstein warped products},
Jordan J. Math. Stat., 5, 85-95, (2012).

\bibitem{GDB} Ghosh, G. C., De, U. C., Binh, T. Q., \textit{Certain curvature
restrictions on a quasi Einstein manifold}, Publ. Math. Debrecen, 69,
209-217, 2006.

\bibitem{Gl} G\l ogowska, M., \textit{On quasi-Einstein Cartan type
hypersurfaces}, J. Geom. Phys. 58(5), 599-614, 2008.

\bibitem{KK} Kim, D. S., Kim, Y., \textit{Compact Einstein warped product
spaces with nonpositive scalar curvature}, Proc. Amer. Math. Soc., 131(8), 2573-2576, 2003.

\bibitem{Oneil} O'Neil, B., \textit{Semi-Riemannian Geometry}, Academic
Press, New York, 1983.

\bibitem{PP} Pahan, S., Pal, B., \textit{On Einstein Sequential Warped
Product Spaces}, Journal of Mathematical Physics, Analysis, Geometry, 15(3),
379-394, 2019.

\bibitem{DPB} Pal, B., Bhattacharyya, A., Dey, S., \textit{Warped product
and quasi-Einstein metrics}, Caspian J. Math. Sciences
(CJMS), 6(1), 1-8, 2017.

\bibitem{SO} Sular, S., \"{O}zg\"{u}r, C., \textit{On quasi-Einstein warped
products}, Annals of the Alexandru Ioan Cuza University-Mathematics, 58(2),
353-362, 2012.

\bibitem{Sahin} \c{S}ahin, B., \textit{Sequential warped product
submanifolds having factors as holomorphic, totally real and pointwise slant}%
, arXiv preprint arXiv:2006.02898, 2020.
\end{thebibliography}
\end{document}